\documentclass[11pt]{article}

\usepackage{amsfonts, amssymb}
\usepackage{amsmath}
\usepackage{amsthm}
\numberwithin{equation}{section}

\newtheorem{theorem}{Theorem}[section]
\newtheorem{lemma}[theorem]{Lemma}

\newcommand{\C}{\mathbb{C}}
\newcommand{\N}{\mathbb{N}}

\begin{document}

\begin{center}
{\bf\LARGE Growth of meromorphic solutions \\
\vskip 1 mm
of delay differential  equations}

\vskip 5 mm

Rod Halburd\footnote{Department of Mathematics, University College London,
Gower Street, London WC1E 6BT, UK. r.halburd@ucl.ac.uk}  and Risto Korhonen\footnote{Department of Physics and Mathematics, University of Eastern Finland, P.O. Box 111,
FI-80101 Joensuu, Finland. risto.korhonen@uef.fi}




\end{center}
\vskip 3mm

\begin{abstract}
Necessary conditions are obtained for certain types of rational 
delay differential equations to admit a non-rational meromorphic solution of hyper-order less than one.
The equations obtained include delay Painlev\'e equations and equations solved by elliptic functions.
\end{abstract}


\section{Introduction}
There have been many studies of the discrete (or difference) Painlev\'e equations.  One way in which difference Painlev\'e equations arise is in the study of
difference equations admitting meromorphic solutions of slow growth in the sense of Nevanlinna.   The idea that the existence of sufficiently many finite-order meromorphic solutions could be considered as a version of the Painlev\'e property for difference equations was first advocated in \cite{ablowitzhh:00}.  This is a very restrictive property, as demonstrated by the relatively short list of possible equations obtained in \cite{halburdk:07} of the form
$w(z+1)+w(z-1)=R(z,w(z))$, where $R$ is rational in $w$ with meromorphic coefficients in $z$, and $w$ is assumed have finite order but to grow  faster than the coefficients.  It was later shown in  \cite{halburdkt:14} that the same list is obtained by replacing the finite order assumption with the weaker assumption of hyper-order less than one.

Some reductions of integrable differential-difference equations are known to yield delay differential equations with formal continuum limits to (differential) Painlev\'e equations.  For example, Quispel, Capel and Sahadevan \cite{quispelcs:92}  obtained the equation
\begin{equation}
\label{cqs-eqn}
w(z)\left[w(z+1)-w(z-1)\right]+aw'(z)=bw(z),
\end{equation}
where $a$ and $b$ are constants,
as a symmetry reduction of the  Kac-van Moerbeke equation.
They showed that equation (\ref{cqs-eqn}) has a formal continuum limit to the first Painlev\'e equation
\begin{equation}
\label{PI}
\frac{{\rm d}^2y}{{\rm d}t^2}=6y^2+t.
\end{equation}
Furthermore, they obtained an associated linear problem for equation (\ref{cqs-eqn}) by extending the symmetry reduction to the Lax pair for the Kac-van Moerbeke equation. 

Painlev\'e-type delay differential equations were also considered in 
Grammaticos, Ramani and Moreira \cite{grammaticosrm:93} from the point of view of a kind of singularity confinement.
More recently, 
Viallet \cite{viallet:14} has introduced a notion of algebraic entropy for such equations.

Most of the present paper is devoted to a proof of the following.
\begin{theorem}\label{mainthm}
Let $w(z)$ be a non-rational meromorphic solution of
    \begin{equation}\label{maineq}
    w(z+1) - w(z-1) + a(z)\frac{w'(z)}{w(z)}=R(z,w(z))=\frac{P(z,w(z))}{Q(z,w(z))},
    \end{equation}
where $a(z)$ is rational, $P(z,w(z))$ is a polynomial in $w(z)$ having rational coefficients in $z$, and $Q(z,0)\not\equiv0$ is a monic polynomial in $w(z)$ with roots that are rational in $z$ and not roots of $P(z,w(z))$. If the hyper-order of $w(z)$  is strictly less than one, then
    \begin{equation}\label{assertion1}
    \deg_w(P)=\deg_w(Q)+1 \leq 3,
    \end{equation}
or the degree of $R(z,w(z))$ as a rational function in $w(z)$ is either $0$ or $1$.
\end{theorem}

If $R(z,w(z))$ does not depend on $w(z)$ then equation (\ref{maineq}) becomes
    \begin{equation}\label{maineq4}
    w(z+1) - w(z-1) + a(z)\frac{w'(z)}{w(z)}=b(z),
    \end{equation}
where $a(z)$ and $b(z)$ are rational. Note that if  $b(z)\equiv p\pi i a(z)$, where $p\in\N$, then $w(z)=C\exp(p\pi i z)$, $C\ne 0$, is a one-parameter family of zero-free entire transcendental finite-order solutions of \eqref{maineq4} for any rational $a(z)$. In the following theorem we will single out the equation \eqref{cqs-eqn} from the class \eqref{maineq4} by introducing an additional assumption that the meromorphic solution has sufficiently many simple zeros.

We will assume that the reader is familiar with the standard notation and basic results of Nevanlinna theory (see, e.g., \cite{hayman:64}).
Let $w(z)$ be a meromorphic function. The hyper-order (or the iterated order) of $w(z)$ is defined by
    $$
    \rho_2(w) = \limsup_{r\to\infty}\frac{\log\log T(r,w)}{\log r},
    $$
where $T(r,w)$ is the Nevanlinna characteristic function of $w$. In value distribution theory the notation $S(r,w)$ usually means a quantity which satisfies $o(T(r,w))$ as $r\to\infty$ outside of an exceptional set of finite linear measure. In what follows we use a slightly modified definition with a larger exceptional set of finite logarithmic measure. We use the notation $N(r,w)$ to denote the integrated counting function of poles counting multiplicities and $\overline{N}(r,w)$ to denote the integrated counting function of poles ignoring multiplicities.

\begin{theorem}\label{0assumptionThm}
Let $w(z)$ be a non-rational meromorphic solution of
equation \eqref{maineq4},
where $a(z)\not\equiv 0$ and $b(z)$ are rational. If the hyper-order of $w(z)$ is strictly less than one and for any $\epsilon>0$
    \begin{equation}\label{simplezerosassumption}
    \overline{N}\left(r,\frac{1}{w}\right) \ge \left(\frac 34+\epsilon\right)T(r,w) + S(r,w),
    \end{equation}
then the coefficients $a(z)$ and $b(z)$ are both constants.
\end{theorem}

Finally, we consider an equation outside the class (\ref{maineq}).
\begin{theorem}\label{0assumptionThmCase2}
Let $w(z)$ be a non-rational meromorphic solution of
    \begin{equation}\label{maineq5}
    w(z+1) - w(z-1) = \frac{a(z)w'(z)+b(z)w(z)}{w(z)^2} + c(z),
    \end{equation}
where $a(z)\not\equiv0$, $b(z)$ and $c(z)$ are rational. If the hyper-order of $w(z)$ is strictly less than one and for any $\epsilon>0$
    \begin{equation}\label{simplezerosassumption2}
    \overline{N}\left(r,\frac{1}{w}\right) \ge \left(\frac 34+\epsilon\right) T(r,w) + S(r,w),
    \end{equation}
then \eqref{maineq5} has the form
        \begin{equation}\label{w22}
    w(z+1) - w(z-1) = \frac{(\lambda+\mu z) w'(z)+(\nu\lambda+ \mu(\nu z-1)) w(z)}{w(z)^2},
    \end{equation}
where $\lambda$, $\mu$ and $\nu$ are constants.
\end{theorem}

When $\mu=\nu=0$ and $\lambda\ne 0$ then equation (\ref{w22}) has a multi-parameter family of elliptic function solutions:
$$
w(z)=\alpha\left[\wp(\Omega z;g_2,g_3)-\wp(\Omega;g_2,g_3)\right],
$$
where $\wp$ is the Weierstrass elliptic function, $\Omega$, $g_2$ and $g_3$ are arbitrary (provided that $\wp'(\Omega;g_2,g_3)\ne 0$ or $\infty$) and
$\alpha^2=-\lambda\Omega/\wp'(\Omega;g_2,g_3)$.  Furthermore, when $\mu=0$, equation \eqref{w22} has a formal continuum limit to the first Painlev\'e equation.
Specifically, we take the limit $\epsilon\to 0$ for fixed $t=\epsilon z$, where $w(z)=1-\epsilon^2y(t)$, $\lambda=2+O(\epsilon)$ and $\lambda\nu=-\frac13\epsilon^5+O(\epsilon^6)$.  Then equation \eqref{w22} becomes 
${\rm d}^3y/{\rm d}t^3=12y\,{\rm d}y/{\rm d}t+1$, which integrates to
${\rm d}^2y/{\rm d}t^2=6y^2+t-t_0$, for some constant $t_0$.  Replacing $t$ with $t+t_0$ gives the first Painlev\'e equation \eqref{PI}.
Finally, when $\mu=0$ and $\lambda\nu\ne 0$, equation \eqref{w22} is a symmetry reduction of the known integrable differential-difference modified Korteweg-de Vries equation
$$
v_t(x,t)=v(x,t)^2\left(v(x+1,t)-v(x-1,t)\right),
$$
in which $v(x,t)=(-2\lambda\nu t)^{-1/2}w(z)$, where $z=x-(2\nu)^{-1}\log t$.

\section{Value distribution of slow growth solutions}
We begin by proving an important lemma, which relates the value distribution of meromorphic solutions of a large class of delay differential equations to the growth of these solutions.
A differential difference polynomial in $w(z)$ is defined by
    $$
    P(z,w) = \sum_{l\in L} b_l(z)w(z)^{l_{0,0}}w(z+c_1)^{l_{1,0}}\cdots w(z+c_\nu)^{l_{\nu,0}}w'(z)^{l_{0,1}}\cdots w^{(\mu)}(z+c_\nu)^{l_{\nu,\mu}},
    $$
where $c_1,\ldots,c_\nu$ are distinct complex constants, $L$ is a finite index set consisting of elements of the form $l=(l_{0,0},\ldots,l_{\nu,\mu})$ and the coefficients $b_l(z)$ are rational functions of $z$ for all $l\in L$.

\begin{lemma}\label{nevlemma}
Let $w(z)$ be a non-rational meromorphic solution of
    \begin{equation}\label{Peq}
    P(z,w)=0
    \end{equation}
where $P(z,w)$ is differential difference polynomial in $w(z)$ with rational coefficients, and let $a_1,\ldots,a_k$ be rational functions satisfying $P(z,a_j)\not\equiv 0$ for all $j\in\{1,\ldots,k\}$. If there exists $s>0$ and $\tau\in(0,1)$ such that
    \begin{equation}\label{n_inequality_lemma}
    \sum_{j=1}^k n\left(r,\frac{1}{w-a_j}\right) \leq k\tau\, n(r+s,w) + O(1),
    \end{equation}
then the hyper-order $\rho_2(w)$ of $w$ is at least $1$.
\end{lemma}

\begin{proof}
We suppose against the conclusion that $\rho_2(w)<1$ aiming to obtain a contradiction. We first show that the assumption $P(z,a_j)\not\equiv 0$ implies that
    \begin{equation}\label{ddemohonko}
    m\left(r,\frac{1}{w-a_j}\right) = S(r,w).
    \end{equation}
This fact is an extension of Mohon'ko's theorem and its difference analogue (see \cite[Remark~5.3]{halburdkt:14}) for differential delay equations with meromorphic solutions of hyper-order strictly less than one.

By substituting $w=g+a_j$ into \eqref{Peq} it follows that
    \begin{equation}\label{Qeq}
    Q(z,g)+R(z)=0,
    \end{equation}
where $R(z)\not\equiv0$ is a rational function, and
    \begin{equation}\label{Q}
    Q(z,g)=\sum_{l\in L} b_l(z) G_l(z,g)
    \end{equation}
is a differential difference polynomial in $g$ such for all $l$ in the finite index set $L$, $G_l(z,g)$ is a non-constant product of derivatives and shifts of $g(z)$.  The coefficients $b_l$ in \eqref{Q} are all rational. Now, letting $E_1=\{\theta\in[0,2\pi):|g(re^{i\theta})|\leq 1\}$ and $E_2=[0,2\pi) \setminus{E_1}$, we have
    \begin{equation}\label{twoparts}
    \begin{split}
     m\left(r,\frac{1}{w-a_j}\right) &=  m\left(r,\frac{1}{g}\right) = \int_{\theta\in E_1}\log^+\left|\frac{1}{g(re^{i\theta})}\right|\frac{d\theta}{2\pi}.
    \end{split}
    \end{equation}
Moreover, for all $z=re^{i\theta}$ such that $\theta\in E_1$,
    \begin{equation*}
    \begin{split}
    \left|\frac{Q(z,g)}{g}\right| & =\frac{1}{|g|} \left| \sum_{l\in L} b_l(z)g(z)^{l_{0,0}}g(z+c_1)^{l_{1,0}}\cdots g(z+c_\nu)^{l_{\nu,0}}g'(z)^{l_{0,1}}\cdots g^{(\mu)}(z+c_\nu)^{l_{\nu,\mu}}\right|\\
    &\leq  \sum_{l\in L} |b_l(z)|\left|\frac{g(z+c_1)}{g(z)}\right|^{l_{1,0}}\cdots\left|\frac{g(z+c_\nu)}{g(z)}\right|^{l_{\nu,0}}\cdot \left|\frac{g'(z)}{g(z)}\right|^{l_{0,1}}\cdots \left|\frac{g^{(\mu)}(z+c_\nu)}{g(z)}\right|^{l_{\nu,\mu}},\\
    \end{split}
    \end{equation*}
since $\deg_g(G_l)\geq 1$ for all $l\in L$ with $l=(l_{0,0},\ldots,l_{\nu,\mu})$. Now, since
    \begin{equation*}
    \begin{split}
    \log^+\left|\frac{1}{g(z)}\right| &\leq \log^+\left|\frac{R(z)}{g(z)}\right| + \log^+\left|\frac{1}{R(z)}\right| \\
    &=    \log^+\left|\frac{Q(z,g)}{g(z)}\right| + \log^+\left|\frac{1}{R(z)}\right|
    \end{split}
    \end{equation*}
by equation \eqref{Qeq}, it follows from \eqref{twoparts} by defining $c_0=0$  that
    \begin{equation}\label{mohonkoend}
    \begin{split}
    &m\left(r,\frac{1}{w-a_j}\right) \leq \int_{\theta\in E_1}\log^+\left|\frac{Q(z,g)}{g(z)}\right|\frac{d\theta}{2\pi} + O(\log r) \\
    & \qquad \leq \sum_{n=0}^\nu \sum_{m=0}^\mu l_{n,m} m\left(r,\frac{g^{(m)}(z+c_n)}{g(z)}\right) + O(\log r) \\
    & \qquad \leq \sum_{n=0}^\nu \sum_{m=0}^\mu l_{n,m} \left(m\left(r,\frac{g^{(m)}(z+c_n)}{g(z+c_n)}\right)+m\left(r,\frac{g(z+c_n)}{g(z)}\right)\right) + O(\log r).
    \end{split}
    \end{equation}
The claim that \eqref{ddemohonko} holds follows by applying the lemma on the logarithmic derivative, its difference analogue \cite[Theorem~5.1]{halburdkt:14} and \cite[Lemma 8.3]{halburdkt:14}, to the right hand side of \eqref{mohonkoend}.

To finish the proof, we observe that from the assumption \eqref{n_inequality_lemma} it follows that
    \begin{equation}\label{Nassumption}
    \sum_{j=1}^k N\left(r,\frac{1}{w-a_j}\right) \leq (\tau+\varepsilon)k\,N(r+s,w) + O(\log r)
    \end{equation}
where $\varepsilon>0$ is chosen so that $\tau+\varepsilon<1$.  The first main theorem of Nevanlinna theory now yields
    \begin{equation}\label{1mt}
    kT(r,w) = \sum_{j=1}^k \left(m\left(r,\frac{1}{w-a_j}\right) + N\left(r,\frac{1}{w-a_j}\right)\right) + O(\log r).
    \end{equation}
By combining \eqref{ddemohonko}, \eqref{Nassumption} and \eqref{1mt} it follows that
    \begin{equation}\label{kT}
    kT(r,w) \leq (\tau+\varepsilon)kN(r+s,w) + S(r,w) \leq (\tau+\varepsilon)k\,T(r+s,w) + S(r,w).
    \end{equation}
An application of \cite[Lemma~8.3]{halburdkt:14} yields $T(r+s,w)=T(r,w)+S(r,w)$, and so \eqref{kT} becomes
    \begin{equation*}
    T(r,w) \leq (\tau+\varepsilon)T(r,w) + S(r,w),
    \end{equation*}
which gives us the desired contradiction $T(r,w)=S(r,w)$ since $\tau+\varepsilon <1$. We conclude that $\rho_2(w)\geq 1$.
\end{proof}

\section{The proof of Theorem \ref{mainthm}}

Suppose that \eqref{maineq} has a non-rational meromorphic solution of hyper-order strictly less than one. Then, by taking the Nevanlinna characteristic function of both sides of \eqref{maineq} and applying an identity due to Valiron \cite{valiron:31} and Mohon'ko \cite{mohonko:71} (see also \cite{laine:93}), we have
    \begin{equation*}
    \begin{split}
    T\left(r,w(z+1) - w(z-1) + a(z)\frac{w'(z)}{w(z)}\right)&=T(r,R(z,w(z)))\\
    & = \deg_w(R(z,w(z))) T(r,w(z)) + O(\log r).
    \end{split}
    \end{equation*}
Thus by using the lemma on the logarithmic derivative and its difference analogue \cite{halburdkt:14}, it follows that
    \begin{equation}\label{Testimate1}
    \begin{split}
    \deg_w(R(z,w(z))) T(r,w(z)) &\leq T\left(r,w(z+1) - w(z-1)\right) +T\left(r,\frac{w'(z)}{w(z)}\right)+O(\log r)\\
    & \leq N\left(r,w(z+1) - w(z-1)\right) + m(r,w(z)) + \overline{N}(r,w(z)) \\ &\qquad + \overline{N}\left(r,\frac{1}{w(z)}\right) + S(r,w).
    \end{split}
    \end{equation}
On using \cite[Lemma~8.3]{halburdkt:14} to obtain
    \begin{equation*}
    \begin{split}
    N\left(r,w(z+1) - w(z-1)\right) &\leq N\left(r,w(z+1)\right) + N\left(r,w(z-1)\right) \\
    &\leq 2 N(r+1,w(z)) = 2N(r,w(z)) + S(r,w),
    \end{split}
    \end{equation*}
inequality \eqref{Testimate1} becomes
    \begin{equation}\label{Testimate2}
    \begin{split}
    \deg_w(R(z,w(z))) T(r,w(z)) &\leq T(r,w(z)) + N(r,w(z)) + \overline{N}(r,w(z)) \\
    &\qquad + \overline{N}\left(r,\frac{1}{w(z)}\right) + S(r,w).
    \end{split}
    \end{equation}
Therefore
    \begin{equation}\label{34T}
    (\deg_w(R(z,w(z)))-3) T(r,w(z)) \leq \overline{N}\left(r,\frac{1}{w(z)}\right) + S(r,w),
    \end{equation}
which implies that $\deg_w(R(z,w(z)))\leq 4$ and furthermore that in the case $\deg_w(R) = 4$ we have $\overline{N}(r,1/w)=T(r,w)+S(r,w)$.

Suppose now that the denominator of $R(z,w(z))$ has at least two distinct non-zero rational roots for $w$ as a function of $z$, say $b_1(z)\not\equiv0$ and $b_2(z)\not\equiv0$. Then we may write equation
 \eqref{maineq} in the form
    \begin{equation}\label{ddeclass2}
    w(z+1) - w(z-1) + a(z)\frac{w'(z)}{w(z)}=\frac{P(z,w(z))}{(w(z)-b_1(z))(w(z)-b_2(z))\widetilde Q(z,w(z))},
    \end{equation}
where $P(z,w(z))\not\equiv0$ and $\widetilde Q(z,w(z))\not\equiv0$ are polynomials in $w(z)$ of at most degree $4$ and $2$, respectively.  We do not exclude the possibility that $\widetilde Q(z,b_1(z))\equiv 0$ or $\widetilde Q(z,b_2(z))\equiv 0$.
We also assume that  $P(z,w(z))$ and $\widetilde Q(z,w(z))$ do not have any common roots. Then neither $b_1(z)$, nor $b_2(z)$ is a solution of \eqref{ddeclass2}, and so they satisfy the first condition of Lemma~\ref{nevlemma}. Assume now that $\hat z\in\C$ is any point where
    \begin{equation}\label{b1point}
    w(\hat z) = b_1(\hat z),
    \end{equation}
and such that none of the rational coefficients of \eqref{ddeclass2} have a zero or a pole at $\hat z$ and $P(\hat z,w(\hat z))\not=0$.
Let $p$ denote the order of the zero of $w-b_1$ at $z=\hat z$.  We will call such a $\hat z$ a {\em generic root of $w-b_1$ of order $p$}.

We will assume, often without further comment, that in similar situations we are only considering generic roots.  
Since the coefficients are rational, when estimating the corresponding unintegrated counting functions, the contribution from the non-generic roots can be included in a bounded error term, leading to an error term of the type $O(\log r)$ in the integrated estimates involving $T(r,w)$. Now, by \eqref{ddeclass2}, it follows that either $w(z+1)$ or $w(z-1)$ has a pole at $z=\hat z$ of order at least $p$. Without loss of generality we may assume that $w(z+1)$ has such a pole at $\hat z$. Then, by shifting the equation \eqref{ddeclass2}, we have
    \begin{equation}\label{ddeclass3}
    \begin{split}
    & w(z+2) - w(z) + a(z+1)\frac{w'(z+1)}{w(z+1)}\\
    &\qquad =\frac{P(z+1,w(z+1))}{(w(z+1)-b_1(z+1))(w(z+1)-b_2(z+1)) \widetilde Q(z+1,w(z+1))},
    \end{split}
    \end{equation}
which implies that $w(z+2)$ has a pole of order one at $z=\hat z$ provided that
    \begin{equation}\label{condition1}
    \deg_w(P) \leq \deg_w(\widetilde  Q) + 2.
    \end{equation}
We suppose first that \eqref{condition1} is valid. By iterating \eqref{ddeclass2} one more step, we have
    \begin{equation}\label{ddeclass4}
    \begin{split}
    &w(z+3) - w(z+1) + a(z+2)\frac{w'(z+2)}{w(z+2)}\\
    &\qquad=\frac{P(z+2,w(z+2))}{(w(z+2)-b_1(z+2))(w(z+2)-b_2(z+2)) \widetilde Q(z+2,w(z+2))}.
    \end{split}
    \end{equation}
Now, if $p>1$ then there must be a pole of order at least $p$ at $w(z+3)$. Hence, in this case, we can pair up the zero of $w-b_1$ at $z=\hat z$ together with the pole of $w$ at $\hat z + 1$ without the possibility of a similar iteration process starting from another point, say $z=\hat z +3$, and resulting in pairing the pole at $\hat z + 1$ with another root of $w-b_1$, or of $w-b_2$. Therefore, we have found a pole of order at least $p$ which can be uniquely associated with the zero of $w-b_1$ at $\hat z$. If, on the other hand, $p=1$ it may in principle be possible that there is another root of $w-b_1$ or of $w-b_2$ at $z=\hat z +3$ which needs to be paired with the pole of $w$ at $z=\hat z + 2$. But since now all of the poles in the iteration are simple, we may still pair up the root of $w-b_1$ at $z=\hat z$ and the pole of $w$ at $z=\hat z +1$. If there is another root of, say, $w-b_1$ at $z=\hat z + 3$ such that $w(\hat z + 4)$ is finite, we can pair it up with the pole of $w$ at $z=\hat z +2$. Thus for any $p\geq 1$ there is a pole of multiplicity at least $p$ which can be paired up with the root of $w-b_1$ at $z=\hat z$.

We can run the above process for roots of $w-b_2$ in completely analogous fashion without any possible overlap in the pairing of poles with the zeros of $w-b_1$ and $w-b_2$. By adding up all points $\hat z$ such that \eqref{b1point} is valid, and similarly for $b_2$, it follows that
    \begin{equation}\label{lemmacondition2val}
    n\left(r,\frac{1}{w-b_1}\right) + n\left(r,\frac{1}{w-b_2}\right) \leq n(r+1,w) + O(1).
    \end{equation}
Therefore the remaining condition \eqref{n_inequality_lemma} of Lemma~\ref{nevlemma} is satisfied, and so $w$ must be of hyper-order at least one by Lemma~\ref{nevlemma}.

We consider now the case where the opposite inequality to \eqref{condition1} holds, i.e.,
    $$
    \deg_w(P) > \deg_w(\widetilde  Q) + 2.
    $$
If $\deg_w(P)=3$, it immediately follows that $\deg_w(Q)=2$, and so the assertion \eqref{assertion1} holds in this case. 
Now assume that
    \begin{equation}\label{condition2}
    4 = \deg_w(P) > \deg_w(\widetilde  Q) + 2 =2
    \end{equation}
and suppose that $\hat z$ is a generic root of $w(z)-b_1(z)$ of order $p$. Then again, by \eqref{ddeclass2}, either $w(z+1)$ or $w(z-1)$ must have a pole at $z=\hat z$ of order at least $p$, and we suppose as above that $w(z+1)$ has the pole at $\hat z$. Then, it follows that $w(z+2)$ has a pole of order $2p$, and $w(z+3)$ a pole of order $4p$ at $z=\hat z$. Hence we can pair the root of $w-b_1$ at $z=\hat z$ and the pole of $w$ at $z=\hat z +1$ the same way as in the case \eqref{condition1}. Identical reasoning holds also for the roots of $w-b_2$, and so \eqref{lemmacondition2val} holds. Lemma~\ref{nevlemma} therefore yields that $w$ is of hyper-order at least one.

Suppose then that
    \begin{equation}\label{condition3}
    4 = \deg_w(P) > \deg_w(\widetilde  Q) + 2 =3,
    \end{equation}
and that $\hat z$ is a point satisfying \eqref{b1point}, and of order $p$. Since now $\deg_w(\widetilde Q)=1$, we may assume without loss of generality that
    \begin{equation*}
    \widetilde Q(z,w(z)) = w(z) - b_3(z),
    \end{equation*}
where $b_3(z)$ is a rational function of $z$. Suppose first that $b_3\not\equiv b_j$ for $j\in\{1,2\}$. Also, it follows by the assumption $Q(z,0)\not\equiv0$ that $b_3\not\equiv 0$. As before, it follows by \eqref{ddeclass2} that either $w(z+1)$ or $w(z-1)$ has a pole of order at least $p$ at $z=\hat z$, and we may again suppose that $w(z+1)$ has that pole. If $p>1$ then \eqref{ddeclass3} implies that $w(z+2)$ has a pole of order $p$, at least, at $z=\hat z$. Even if $w-b_j$ has a root at $z=\hat z+3$ for some $j\in\{1,2,3\}$, we have found at least one pole for each root of $w-b_j$ in this iteration sequence, taking multiplicities into account. Hence we can pair the root of $w-b_1$ at $z=\hat z$ and the pole of $w$ at $z=\hat z +1$ the same way as in cases \eqref{condition1} and \eqref{condition2}. However, if $p=1$ it may in principle be possible that the pole of the right hand side of \eqref{ddeclass3} at $z=\hat z$ cancels with the pole of the term
    \begin{equation*}
    a(z+1)\frac{w'(z+1)}{w(z+1)}
    \end{equation*}
at $z=\hat z$ in such a way that $w(\hat z +2)$ remains finite. If $w(\hat z +2)\not= b_j(\hat z)$ for $j\in\{1,2,3\}$, then it follows from \eqref{ddeclass4} that $w(z+3)$ has a pole at $z=\hat z$, and we can pair up the root of $w-b_1$ at $z=\hat z$ and the pole of $w$ at $z=\hat z +1$. If $w(\hat z +2)= b_j(\hat z)$ for some $j\in\{1,2,3\}$, it may happen that also $w(\hat z+3)$ stays finite. If all points $\hat z$ such that
    \begin{equation*}
    w(\hat z) = b_j(\hat z)
    \end{equation*}
are a part of an iteration sequence of this form, i.e., that
    \begin{equation*}
    w(\hat z) = b_{j_1}(\hat z), \qquad w(\hat z + 1) = \infty, \qquad w(\hat z) = b_{j_2}(\hat z), \qquad j_1,j_2\in\{1,2,3\},
    \end{equation*}
by adding up all roots of $w-b_j$, $j\in\{1,2,3\}$, we still have the inequality
    \begin{equation*}
    n\left(r,\frac{1}{w-b_1}\right) + n\left(r,\frac{1}{w-b_2}\right) + n\left(r,\frac{1}{w-b_3}\right) \leq 2 n(r+1,w) + O(1).
    \end{equation*}
Also, we have already noted that neither $b_1$, nor $b_2$ satisfy the equation \eqref{ddeclass2}. The same is true also for $b_3$, and so all conditions of Lemma~\ref{nevlemma} are satisfied. Hence the hyper-order of $w$ is at least one also in the case \eqref{condition3}.

Suppose now that the denominator of $R(z,w(z))$ in \eqref{maineq} has at least one non-zero rational root, say $b_1(z)\not\equiv0$. Then \eqref{maineq} can be written as
    \begin{equation}\label{ddeclassv2}
    w(z+1) - w(z-1) + a(z)\frac{w'(z)}{w(z)}=\frac{P(z,w(z))}{(w(z)-b_1(z))^n \check{Q}(z,w(z))},
    \end{equation}
where $P(z,w(z))\not\equiv0$ and $(w(z)-b_1(z))^n\check{Q}(z,w(z))$ are polynomials in $w(z)$ of at most degree $4$ and $n+m\leq 4$, respectively, and without common roots. Then $b_1(z)$ is not a solution of \eqref{ddeclass2}, and thus the first condition of Lemma~\ref{nevlemma} is satisfied for $b_1$. Assume first that $n\in\{2,3,4\}$, and suppose that $\hat z$ is generic root of $w(z)-b_1(z)$ of order $p$. 
Then either $w(z+1)$ or $w(z-1)$ has a pole of order $np$ at least, at $z=\hat z$, and we suppose without loss of generality that $w(\hat z+1)=\infty$ is such a pole. Suppose next that
    \begin{equation}\label{v2cond}
    \deg_w(P)\leq n+m.
    \end{equation}
Then $w(\hat z+2)$ is a pole of order one, and $w(\hat z+3)$ a pole of order $np$, at least. By continuing the iteration, it follows that $w(\hat z+4)$ is again a simple pole or a finite value. Therefore it may be that $w(\hat z+4) = b_1(\hat z+4)$, and so it is at least in principle possible that $w(\hat z+5)$ is a finite value. But even so, by adding up all roots of $w-b_1$ and poles of $w$ in the set $\{\hat z,\ldots,\hat z + 4\}$, and taking into account multiplicities of these points, we find that there are at least $2np+1$ poles for $2p$ roots of $w-b_1$. This is the ``worst case scenario'' in the sense that if $w(\hat z+4) \not= b_1(\hat z+4)$, or a root of $\check{Q}(z,w(z))$, then $w(\hat z+5)$ is a pole of order $np$, and we have even more poles for every root of $w-b_1$. By adding up the contribution from all points $\hat z$ to corresponding counting functions, it follows that
    \begin{equation*}
    n\left(r,\frac{1}{w-b_1}\right) \leq \frac{1}{n} n(r+4,w) + O(1).
    \end{equation*}
Thus both conditions of Lemma~\ref{nevlemma} are satisfied, and so the hyper-order of $w$ is at least one.

Assume now that
    \begin{equation}\label{v2cond2}
    \deg_w(P)\geq n+m+1.
    \end{equation}
Suppose again that $\hat z$ is a generic root of $w(z)-b_1(z)$ of order $p$. Then, as in the case \eqref{v2cond} either $w(\hat z+1)$ or $w(\hat z-1)$, say $w(\hat z+1)$, is a pole of order $np$ at least. This implies that $w(\hat z+2)$ is a pole of order $np$ at least, and so, the only way that $w(\hat z+4)$ can be finite is that $w(\hat z+3)=b_1(\hat z + 3)$, or $w(\hat z+3)$ is a root of $\check{Q}(z,w(z))$, with multiplicity $p$. Even if this would be the case, we have found $2np$ poles, taking into account multiplicities, that correspond uniquely to $2p$ roots of $w-b_1$. Therefore, we have
    \begin{equation*}
    n\left(r,\frac{1}{w-b_1}\right) \leq \frac{1}{n} n(r+3,w) + O(1)
    \end{equation*}
by going through all roots of $w-b_1$ in this way. Lemma~\ref{nevlemma} thus implies that the hyper-order of $w$ is at least one.

Suppose now that $Q(z,w)$ in the equation \eqref{maineq} has only one simple root, and assume first that
    \begin{equation}\label{n1cond}
    \deg_w(P)\geq 3.
    \end{equation}
We can therefore write the denominator of the right hand side of \eqref{maineq} in the form $Q(z,w)=w-b_1$. Let $\hat z$ be a generic root of $w(z)-b_1(z)$ of order $p$. Then, either $w(\hat z+1)$ or $w(\hat z-1)$ is a pole of order $p$ at least. We assume again without loss of generality that $w(\hat z+1)$ is a pole of order $p$. Then $w(\hat z+2)$ is a pole of order $2p$ at least, and $w(\hat z+3)$ is a pole of order $4p$, and so on. In this case we therefore have
    \begin{equation*}
    n\left(r,\frac{1}{w-b_1}\right) \leq \frac{1}{3} n(r+2,w) + O(1).
    \end{equation*}
Lemma~\ref{nevlemma} thus implies that the hyper-order of $w$ is at least one.

Assume now that $Q(z,w)$ in \eqref{maineq} has only one simple root, and
    \begin{equation}\label{n1cond2}
    \deg_w(P) \leq 2.
    \end{equation}
If $\deg_w(P) =2$, then $\deg_w(P)=\deg_w(Q)+1$ and thus the assertion \eqref{assertion1} holds. If $\deg_w(P) \leq 1$, then $\deg_w(R)=1$.

The final remaining case is the one where $R(z,w(z))$ is polynomial in $w(z)$. Then \eqref{ddeclass2} takes the form
    \begin{equation}\label{ddeclassv3}
    w(z+1) - w(z-1) + a(z)\frac{w'(z)}{w(z)}=P(z,w(z)),
    \end{equation}
where the degree of $P(z,w(z))$ is at most $4$. If $\deg_w(P)=1$, then \eqref{assertion1} holds, and if $\deg_w(P)=0$, it follows that $R(z,w)$ in \eqref{maineq} is a polynomial of degree $0$ as asserted. Assume therefore that $\deg_w(P)\geq 2$, and suppose first that $w(z)$ has either infinitely many zeros or poles (or both). Suppose that there is a pole or a zero of $w(z)$ at $z=\hat z$. Then either there is a cancelation with one of the coefficients, or $w(z)$ has a pole of order at least $1$ at $z=\hat z +1$, or at $z=\hat z - 1$. Since the coefficients of \eqref{ddeclassv3} are rational, we can always choose a zero or a pole of $w(z)$ in such a way that there is no cancelation with the coefficients. Suppose, without loss of generality, that there is a pole of $w(z)$ at $z=\hat z + 1$. By shifting \eqref{ddeclassv3} up, it follows that $w(z)$ has a pole of order $\deg_w(P)$, at least, at $z=\hat z + 2$, and a pole of order $(\deg_w(P))^2$ at $z=\hat z + 3$, and so on. The only way this string of poles with exponential growth in the multiplicity can terminate, or there can be a drop in the orders of poles, is that there is a cancelation with a suitable zero of a coefficient of \eqref{ddeclassv3}. But since the coefficients are rational and thus have finitely many zeros, and $w(z)$ has infinitely many zeros or poles, we can choose the starting point $\hat z$ of the iteration from outside a sufficiently large disc in such a way that no cancelation occurs. Thus,
    \begin{equation*}
    n(d+|\hat z|,w) \geq (\deg_w(P))^d
    \end{equation*}
for all $d\in \N$, and so
    \begin{equation*}
    \begin{split}
    \lambda_2(1/w) &= \limsup_{r\to\infty} \frac{\log\log n(r,w)}{\log r} \\
    &\geq \limsup_{d\to\infty} \frac{\log\log n(d+|\hat z|,w)}{\log (d+|\hat z|)}\\
    &\geq \limsup_{d\to\infty} \frac{\log\log (\deg_w(P))^d}{\log (d+|\hat z|)}=1.
    \end{split}
    \end{equation*}
Therefore, $\rho_2(w)\geq \lambda_2(1/w) \geq 1$.

Suppose now that $w(z)$ has finitely many poles and zeros, and that $\rho_2(w)<1$. Then
    \begin{equation}\label{w}
    w(z)=f(z)\exp(g(z)),
    \end{equation}
where $f(z)$ is a rational function and $g(z)$ is entire. By substituting \eqref{w} into \eqref{ddeclassv3}, it follows that
    \begin{equation}\label{reg}
    f(z+1)e^{g(z+1)} - f(z-1)e^{g(z-1)} + a(z)\left(\frac{f'(z)}{f(z)}+g'(z)\right) = P(z,f(z)\exp(g(z))).
    \end{equation}
Now, since $\rho_2(\exp(g(z)))<1$, it follows by the difference analogue of the lemma on the logarithmic derivatives, \cite{halburdkt:14}, that
    \begin{equation*}
    T\left(r,e^{g(z+1)-g(z)}\right)=m\left(r,e^{g(z+1)-g(z)}\right) = S(r,e^{g}),
    \end{equation*}
and similarly
    \begin{equation*}
    T\left(r,e^{g(z-1)-g(z)}\right)=m\left(r,e^{g(z-1)-g(z)}\right) = S(r,e^{g}).
    \end{equation*}
Hence, by writing \eqref{reg} in the form
    \begin{equation*}
    \begin{split}
    & e^{g(z)}\left(f(z+1)e^{g(z+1)-g(z)} - f(z-1)e^{g(z-1)-g(z)}\right) + a(z)\left(\frac{f'(z)}{f(z)}+g'(z)\right) \\
    \qquad &= P(z,f(z)\exp(g(z))),
    \end{split}
    \end{equation*}
and taking Nevanlinna characteristic from both sides, we arrive at the equation
    $$
    \deg_w(P) T(r,e^{g}) = T(r,e^{g}) + S(r,e^{g}) + O(\log r).
    $$
Since $\deg_w(P)\geq 2$ by assumption, this implies that $g$ is a constant. But this means that $w$ is rational, which is a contradiction. Thus $\rho_2(w) \geq 1$.

\section{The proof of Theorem \ref{0assumptionThm}}

If $z=\hat z$ is a zero of $w(z)$, then by \eqref{maineq4} there is a pole of $w(z)$ at $z=\hat z + 1$ or at $z=\hat z - 1$ (or at both points) unless
    \begin{equation}\label{abrational}
    a(\hat z)=0 \qquad  \textrm{or} \qquad b(\hat z)=\infty.
    \end{equation}
Since $a(z)$ and $b(z)$ are rational, there are only finitely many points such that \eqref{abrational} holds. By \eqref{simplezerosassumption} and by the assumption that $w(z)$ is non-rational, it follows that $w(z)$ has infinitely many zeros. Suppose that $\hat z$ is a zero of $w(z)$ such that \eqref{abrational} does not hold. We need to consider two cases.

Suppose first that there is a pole of $w(z)$ at both points $z=\hat z - 1$ and $z=\hat z + 1$. Then, from \eqref{maineq4} it follows that there are poles of $w(z)$ at $z=\hat z - 2$ and $z=\hat z + 2$. Now, at least in principle we may have $w(\hat z -3)=0=w(\hat z + 3)$. Hence, in this case we can find at least four poles of $w(z)$ (ignoring multiplicity) which correspond to three zeros (also ignoring multiplicity) of $w(z)$ and to no other zeros.

Assume now that there is a pole of $w(z)$ at only one of the points $z=\hat z +1$ and $z=\hat z - 1$. Without loss of generality we can then suppose that $w(z)$ has a pole at $z=\hat z+1$ (the case where the pole is at $z=\hat z-1$ is completely analogous). 
Let $w(z)$ have a pole of order $p$ at $z=\hat z$. Then,
    \begin{equation}\label{pseq}
    \begin{split}
    &  w(z-1) = K + O(z-\hat z), \qquad K\in \C, \\
    & w(z) = \alpha(z-\hat z)^p + O((z-\hat z)^{p+1}), \qquad  \alpha\in\C\setminus\{0\}\\
    & w(z+1) = -\frac{p a(z)}{z-\hat z} + O(1), \\
    & w(z+2) = \frac{a(z+1)}{z-\hat z} + O(1),\\
    & w(z+3) = \frac{a(z+2)-p a(z)}{z-\hat z}  + O(1)
    \end{split}
    \end{equation}
in a neighborhood of $\hat z$.  If $a(\hat z+2)-p a(\hat z)\ne 0$ then there are at least three poles of $w(z)$ for every two zeros of $w(z)$ (ignoring multiplicity) in the sequence \eqref{pseq}, even if $w(\hat z+4)=0$.
Since $a(z)\not\equiv 0$ is rational, $a(z+2)-p a(z)$ can only vanish at an infinite number of points if $p=1$ and $a(z)=a$ is a constant.

Suppose now that $w(z)$ has a simple zero at $z=\hat z$ and is analytic at $z=\hat z-1$. 
In this case the sequence of iterates in \eqref{pseq} becomes
    \begin{equation}
    \begin{split}
    &  w(z-1) = K + O(z-\hat z), \qquad K\in \C, \\
    & w(z) = \alpha(z-\hat z) + O((z-\hat z)^2),\qquad  \alpha\in\C\setminus\{0\}\\
    & w(z+1) = -\frac{a}{z-\hat z}  + b(z) + K + O(z-\hat z), \\
    & w(z+2) = \frac{a}{z-\hat z} + b(z+1)+b(z)+K + O(z-\hat z),\\
    & w(z+3) = b(z+2)-b(z+1) + O(z-\hat z).
    \end{split}
    \end{equation}
If $w(\hat z + 3)\not=0$, then there are two poles (at $z=\hat z + 1$ and $z=\hat z + 2$) in this sequence that can be uniquely grouped with the zero of $w(z)$ at $z=\hat z$. This can only be avoided if $b(z)=b$ is also a constant.
If $a$ and $b$ are not both constants then
    \begin{equation*}
    \overline{n}\left(r,\frac{1}{w}\right) \leq \frac34 \overline{n}(r+1,w) + O(1).
    \end{equation*}
Hence, for any $\varepsilon>0$,
    \begin{equation*}
    \overline{N}\left(r,\frac{1}{w}\right) \leq \left(\frac34+\frac\varepsilon 2\right) \overline{N}(r+1,w) + O(\log r),
    \end{equation*}
and so by using \cite[Lemma~8.3]{halburdkt:14} to deduce that $\overline{N}(r+1,w)=\overline{N}(r,w) + S(r,w)$, we have
    \begin{equation*}
    \overline{N}\left(r,\frac{1}{w}\right) \leq \left(\frac34+\frac\varepsilon 2\right) T(r,w) + S(r,w),
    \end{equation*}
    which contradicts the assumption of the theorem.

\section{The proof of Theorem \ref{0assumptionThmCase2}}

By \eqref{simplezerosassumption2} and by the assumption that $w(z)$ is non-rational, it follows that $w(z)$ has infinitely many zeros. If $z=\hat z$ a zero of $w(z)$ of order $p$, then by \eqref{maineq5} there is a pole of $w(z)$ of order $p+1$, at least, at $z=\hat z + 1$ or at $z=\hat z - 1$ (or at both points) unless
    \begin{equation}\label{abrational2}
    a(\hat z)=0 \qquad  \textrm{and} \qquad b(\hat z)=0, \qquad \textrm{or} \qquad c(\hat z)=\infty.
    \end{equation}
Since $a(z)$, $b(z)$ and $c(z)$ are rational, there are only finitely many points such that \eqref{abrational2} holds. Suppose that $\hat z$ is a zero of $w(z)$ such that \eqref{abrational2} does not hold. We need to consider two cases.

Suppose first that there is a pole of $w(z)$ at both points $z=\hat z - 1$ and $z=\hat z + 1$. Then, even if there are zeros of $w(z)$ at both $z=\hat z - 2$ and $z=\hat z + 2$, we can group together three zeros of $w$ (ignoring multiplicity) with at least four poles of $w$ (counting multiplicity).

Assume now that there is a pole of $w(z)$ at only one of the points $z=\hat z +1$ and $z=\hat z - 1$. Without loss of generality we can then suppose that $w(z)$ has a pole at $z=\hat z+1$ (the case where the pole is at $z=\hat z-1$ is completely analogous). Consider first the case where the zero is simple, and suppose that $c(z)\not\equiv0$. Then, in a neighborhood of $\hat z$,
    \begin{equation}\label{seq1}
    \begin{split}
    &  w(z-1) = K + O(z-\hat z), \qquad K\in \C, \\
    & w(z) = \alpha(z-\hat z) + O((z-\hat z)^2), \qquad  \alpha\in\C\setminus\{0\}\\
    & w(z+1) = \frac{a(z)}{\alpha(z-\hat z)^2}  + \frac{b(z)}{\alpha(z-\hat z)} + c(z) +  K + O(z-\hat z), \\
    & w(z+2) = c(z+1) + O(z-\hat z),\\
    & w(z+3) = \frac{a(z)}{\alpha(z-\hat z)^2}  + \frac{b(z)}{\alpha(z-\hat z)}  + O(1),
    \end{split}
    \end{equation}
where there can be at most finitely many $\hat z$ such that $c(\hat z+1)=0$. Hence there are two poles of $w(z)$ (counting multiplicity) corresponding to one zero (ignoring multiplicity) in this case.

Assume now that $c(z)\equiv 0$, $w(z)$ has a pole at $z=\hat z+1$, and that $w(\hat z - 1)$ is finite. Then, in a neighborhood of $\hat z$,
    \begin{equation}\label{seq2}
    \begin{split}
    &  w(z-1) = K + O(z-\hat z), \qquad K\in \C, \\
    & w(z) = \alpha(z-\hat z) + O((z-\hat z)^2), \qquad  \alpha\in\C\setminus\{0\}\\
    & w(z+1) = \frac{a(z)}{\alpha(z-\hat z)^2}  + \frac{b(z)}{\alpha(z-\hat z)} + O(1), \\
    & w(z+2) = \alpha\left(1-\frac{2a(z+1)}{a(z)}\right)(z-\hat z)+ O((z-\hat z)^2),\\
    & w(z+3) = \frac{a(z)(a(z+2)-2a(z+1)+a(z))}{(a(z)-2a(z+1))\alpha(z-\hat z)^2}  
   +\frac{\gamma(z)}{ \alpha (z-\hat z)} + O(1),
    \end{split}
    \end{equation}
    where 
 \begin{equation}\label{gammadef}
    \begin{split}
    \gamma(z) = & \frac{a(z)b(z+2)-(2a(z+1)-a(z))b(z)}{a(z)-2a(z+1)} \\
  &\quad  -
    \frac{2a(z+2)[a(z)a'(z+1)-a(z+1)a'(z)]}{(a(z)-2a(z+1))^2}.
    \end{split}
\end{equation}
If $w(\hat z + 3)$ is a pole of order two, then are at least four poles (counting multiplicities) in this sequence that can be uniquely grouped with the two zeros of $w(z)$ (ignoring multiplicities). The only way that $w(z)$ can have a simple pole at $z=\hat z +3$ is that
    \begin{equation}\label{asmalleq2}
    a(\hat z+2)-2a(\hat z+1)+a(\hat z) = 0
    \end{equation}
    and $\gamma(z)\not\equiv 0$.
But in this case from equation \eqref{maineq5} it follows that
    \begin{equation*}
    w(z + 4) = -\frac{\alpha a(z+3)}{\gamma(z)} +  O(z-\hat z)
    \end{equation*}
for all $z$ in a neighborhood of $\hat z$, and so $w(\hat z + 4)$ is finite and non-zero with at most finitely many exceptions. Thus we can group together three poles of $w(z)$ (counting multiplicities) and two zeros of $w(z)$ (ignoring multiplicities). The only way that $w(\hat z + 3)$ can be finite is that \eqref{asmalleq2} holds together with $\gamma\equiv 0$.

If the order of the zero of $w(z)$ at $z=\hat z$ is $p\geq 2$, then there are always at least three poles of $w(z)$ (counting multiplicity) for each two zeros of $w(z)$ (ignoring multiplicity) in sequence \eqref{seq1} and \eqref{seq2}.

If there are only finitely many zeros $\hat z$ of $w(z)$ such that \eqref{asmalleq2} and $\gamma(z)\equiv0$ both hold, then
    \begin{equation*}
    \overline{n}\left(r,\frac{1}{w}\right) \leq \frac34 n(r+1,w) + O(1).
    \end{equation*}
Hence, for any $\varepsilon>0$,
    \begin{equation*}
    \overline{N}\left(r,\frac{1}{w}\right) \leq \left(\frac34+\frac\varepsilon 2\right) N(r+1,w) + O(\log r),
    \end{equation*}
and so by using \cite[Lemma~8.3]{halburdkt:14} to deduce that $N(r+1,w)=N(r,w) + S(r,w)$, we have
    \begin{equation*}
    \overline{N}\left(r,\frac{1}{w}\right) \leq \left(\frac34+\frac\varepsilon 2\right) T(r,w) + S(r,w).
    \end{equation*}
This is in contradiction with \eqref{simplezerosassumption2}, and so there must be infinitely many points $\hat z$ such that \eqref{asmalleq2} and $\gamma(z)\equiv0$  are both satisfied. The only rational functions $a(z)$ satisfying \eqref{asmalleq2} at infinitely many points have the form $a(z)=\lambda + \mu z$, for some constants $\lambda$ and $\mu$.  
Equation $\gamma(z)\equiv0$  becomes
$$
\frac{b(z+2)}{a(z+2)}-\frac{b(z)}{a(z)}=2\frac{a(z+1)a'(z)-a(z)a'(z+1)}{a(z)a(z+2)}=\mu\left(\frac{1}{a(z)}-\frac 1{a(z+2)}\right).
$$
Hence
$b(z)=ka(z)-\mu$,
where $k$ is a constant (since $a$ and $b$ are assumed to be rational).  

\vskip 5 mm 

\noindent{\bf\Large Acknowledgements}
\vskip 3mm
The first author was partially supported by EPSRC grant EP/K041266/1.
The second author was partially supported by the Academy of Finland grants (\#286877) and (\#268009).

\def\cprime{$'$}
\providecommand{\bysame}{\leavevmode\hbox to3em{\hrulefill}\thinspace}
\providecommand{\MR}{\relax\ifhmode\unskip\space\fi MR }
\providecommand{\MRhref}[2]{%
  \href{http://www.ams.org/mathscinet-getitem?mr=#1}{#2}
}
\providecommand{\href}[2]{#2}


\begin{thebibliography}{1}

\bibitem{ablowitzhh:00}
M.~J. Ablowitz,  R. Halburd, and B. Herbst,
\emph{On the extension of the {P}ainlev\'e property to difference
              equations},
Nonlinearity {\bf 13} (2000), 889--905.

\bibitem{grammaticosrm:93}
B.~Grammaticos, A.~Ramani, and I.~C. Moreira, \emph{Delay-differential
  equations and the {P}ainlev{\'e} transcendents}, Physica A \textbf{196}
  (1993), 574--590.
  
\bibitem{halburdk:07} R.~G. Halburd and R.~J.~Korhonen, \emph{Finite-order meromorphic solutions and the discrete Painlev\'e equations}, Proc. Lond. Math. Soc. \textbf{94} (2007), 443--474.

\bibitem{halburdkt:14}
R.~G. Halburd, R.~Korhonen, and K.~Tohge, \emph{Holomorphic curves with
  shift-invariant hyperplane preimages}, Trans. Amer. Math. Soc. \textbf{366}
  (2014), no.~8, 4267--4298.
  
\bibitem{hayman:64}
W.~K. Hayman, \emph{Meromorphic functions}, Clarendon Press, Oxford, 1964.

\bibitem{laine:93}
I.~Laine, \emph{{N}evanlinna theory and complex differential equations}, Walter
  de Gruyter, Berlin, 1993.

\bibitem{mohonko:71}
A.~Z. Mohon'ko, \emph{The {N}evanlinna characteristics of certain meromorphic
  functions}, Teor. Funktsii Funktsional. Anal. i Prilozhen \textbf{14} (1971),
  83--87, (Russian).

\bibitem{quispelcs:92}
G.~R.~W. Quispel, H.~W. Capel, and R.~Sahadevan, \emph{Continuous symmetries of
  differential-difference equations: the {Kac}-van {M}oerbeke equation and
  {P}ainlev{\'e} reduction}, Phys. Lett. A \textbf{170} (1992), 379--383.

\bibitem{valiron:31}
G.~Valiron, \emph{Sur la d{\'e}riv{\'e}e des fonctions alg{\'e}bro{\"i}des},
  Bull. Soc. Math. France \textbf{59} (1931), 17--39.

\bibitem{viallet:14}
C.-M. Viallet, \emph{Algebraic entropy for differential-delay equations},
  arXiv:1408.6161 (2014).

\end{thebibliography}
\end{document}